\documentclass[11pt]{iopart}
\usepackage{amsthm,amssymb,bbm,graphicx,psfrag,color,enumerate}
\usepackage{algorithm}
\usepackage{algorithmic}
\usepackage{graphicx}
\usepackage{yfonts}
\usepackage{subfloat}
\usepackage{subfigure}

\newtheorem{thm}{Theorem}[section]

\def\norm#1{\hspace{0.2ex} \|#1\| \hspace{0.2ex}}

\newcommand{\R}{\ensuremath{\mathbbm{R}}}

\newcommand{\kommentar}[1]{}

\usepackage{tikz}
\usepackage{tikz-3dplot}

\begin{document}


\title[Monotonicity-based regularization for EIT]{Monotonicity-based regularization for phantom experiment data in Electrical Impedance Tomography}

\author{Bastian Harrach${}^\dag$ and Mach Nguyet Minh${}^\ddag$}
\address{\dag\ Department of Mathematics,
Goethe University Frankfurt, Germany}
\address{\ddag\ Department of Mathematics and Statistics,
University of Helsinki, Finland}
\ead{\mailto{harrach@math.uni-frankfurt.de}, \mailto{minh.mach@helsinki.fi}}

\begin{abstract} In electrical impedance tomography, algorithms based on minimizing the linearized-data-fit residuum have been widely used due to their real-time implementation and satisfactory reconstructed images. However, the resulting images usually tend to contain ringing artifacts. In this work, we shall minimize the linearized-data-fit functional with respect to a linear constraint defined by the monotonicity relation in the framework of real electrode setting. Numerical results of standard phantom experiment data confirm that this new algorithm improves the quality of the reconstructed images as well as reduce the ringing artifacts.  
\end{abstract}





\section{Introduction}
\label{Sec:intro}

Electrical Impedance Tomography (EIT) is a recently developed non-invasive imaging technique, where the inner structure of a reference object can be recovered from the current and voltage measurements on the object's surface. It is fast, inexpensive, portable and requires no ionizing radiation. For these reasons, EIT qualifies for continuous real time visualization right at the bedside. 

In clinical EIT applications, the reconstructed images are usually obtained by minimizing the linearized-data-fit residuum \cite{cheney1990noser,adler2009greit}. These algorithms are fast and simple. However, to the best of the authors' knowledge, there is no rigorous global convergence results that have been proved so far. Moreover, the reconstructed images usually tend to contain ringing artifacts. 

Recently, Seo and one of the author have shown in \cite{harrach2010exact} that a single linearized step can give the correct shape of the conductivity contrast. This result raises a question that whether to regularize the linearized-data-fit functional such that the corresponding minimizer yields a good approximation of the conductivity contrast. An affirmative answer has been proved in \cite{harrach2015enhancing} for the continuum boundary data. In the present paper, we shall apply this new algorithm to the real electrode setting and test with standard phantom experiment data. Numerical results later on show that this new algorithm helps to improve the quality of the reconstructed images as well as reduce the ringing artifacts. It is worth to mention that our new algorithm is non-iterative, hence, it does not depend on an initial guess and does not require expensive computation. Other non-iterative algorithms, for example, the Factorization Method  \cite{scherzer2011handbook,harrach2013recent} and the Monotonicity-based Method \cite{tamburrino2002new,tamburrino2006monotonicity,harrach2013monotonicity,aykroyd2006conditional}, on the other hand, are much more sensitive to measurement errors  than our new algorithm when phantom data or real data are applied \cite{azzouz2007factorization,harrach2010factorization,zhou2015monotonicity,garde2015convergence}.

The paper is organized as follows. In Section \ref{Sec:setting} we introduce the mathematical setting, describe how the measured data can be collected and set up a link between the mathematical setting and the measured data. Section \ref{Sec:algorithm} presents our new algorithm and the numerical results were shown in Section \ref{Sec:num}. We conclude this paper with a brief discussion in Section 

\section{Mathematical setting}
\label{Sec:setting}

Let $\Omega \subseteq \R^n, n\ge 2$ describe the imaging subject and $\sigma: \Omega \to \R$ be the unknown conductivity distribution inside $\Omega$. We assume that $\Omega$ is a bounded domain with smooth boundary $\partial \Omega$ and that the function $\sigma$ is real-valued, strictly positive and bounded. Electrical Impedance Tomography (EIT) aims at recovering $\sigma$ using voltage and current measurements on the boundary of $\Omega$. There are several ways to inject currents and measure voltages. We shall follow the \textit{Neighboring Method} (aka Adjacent Method) which was suggested by Brown and Segar in 1987 \cite{brown1987sheffield} and is still widely being used by practitioners. In this method, electrodes are attached on the object's surface, and an electrical current is applied through a pair of adjacent electrodes whilst the voltage is measured on all other pairs of adjacent electrodes excluding those pairs containing at least one electrode with injected current. Figure \ref{fig:M1} illustrates the first and second current patterns for a $16$-electrode EIT system. At the first current pattern (figure \ref{fig:M1}a), small currents of intensity $I^{(1)}_1$ and $I^{(1)}_2=-I^{(1)}_1$ are applied through electrodes $E_1$ and $E_2$ respectively, and the voltage differences $U^{(1)}_{3}, U^{(1)}_{4},\dots, U^{(1)}_{15}$ are measured successively on electrode pairs $(E_3,E_4),(E_4,E_5),\ldots,(E_{15},E_{16})$. In general, for a $L$-electrode EIT system, at the $k$-th current pattern, by injecting currents $I^{(k)}_k$ and $I^{(k)}_{k+1}=-I^{(k)}_k$ to electrodes $E_k$ and $E_{k+1}$ respectively, one gets $L-3$ voltage measurements $\{U^{(k)}_{l}\}$, where $l\in \{1,2,\dots, L\}$ and $|k-l|> 1$.
Note that here and throughout the paper, the electrode index is always considered modulo $L$, i.e. the index $L+1$ also refers to the first electrode, etc.

\begin{figure}[h!]
\centering

\tikzstyle{place}=[circle, draw, inner sep=0mm,minimum size=0.5cm,font=\fontsize{2}{2.5}\selectfont]

\begin{tikzpicture}

\def\xl{-3}
\def\yl{4}
\def\rs{1.7}
\def\rb{1.8}
\def\rbb{2.2}
\coordinate (C) at (\xl,\yl);

\draw (\xl,\yl) circle (\rs);
\node (omega) at (-3, 4) {\Large$\Omega$};
\node (text) at (-5, 6) {a)};

\foreach \a in {0,22.5,45,67.5,90,112.5,135,157.5,180,202.5,225,247.5,270,292.5,315,337.5}
{
\pgfmathsetmacro{\xss}{\xl+ \rs*cos(\a-4)}
\pgfmathsetmacro{\xsb}{\xl+ \rb*cos(\a-4)}
\pgfmathsetmacro{\xbs}{\xl+ \rs*cos(\a+4)}
\pgfmathsetmacro{\xbb}{\xl+ \rb*cos(\a+4)}

\pgfmathsetmacro{\yss}{\yl+ \rs*sin(\a-4)}
\pgfmathsetmacro{\ysb}{\yl+ \rb*sin(\a-4)}
\pgfmathsetmacro{\ybs}{\yl+ \rs*sin(\a+4)}
\pgfmathsetmacro{\ybb}{\yl+ \rb*sin(\a+4)}

\draw[fill=black]
(\xss,\yss)--(\xsb,\ysb)--(\xbb,\ybb)--(\xbs,\ybs)--cycle;
}

\foreach \a/\atext in {
0/E_1,
22.5/E_2,
45/E_3,
67.5/E_4,
90/E_5,
112.5/E_6,
135/E_7,
157.5/E_8,
180/E_9,
202.5/E_{10},
225/E_{11},
247.5/E_{12},
270/E_{13},
292.5/E_{14},
315/E_{15},
337.5/E_{16}}
{
\coordinate (T) at ($(C)+(\a:1.33)$);
\draw (T) node {\tiny{$\atext$}};
}

\shade[line width=2pt, top color=black] 
  (-3.5, 5) to [out=200, in=120]   (-2.2,4.3)
         to [out=20, in=50] (-3.5,5) ; 
         
\shade[line width=2pt, top color=black] 
  (-3.5, 3) to [out=20, in=10]   (-2.5,3.8)
         to [out=200, in=150] (-3.5,3) ; 
         
                 
\foreach \a in {0,22.5,45,67.5,90,112.5,135,157.5,180,202.5,225,247.5,270,292.5,315,337.5}
{
\pgfmathsetmacro{\xs}{\xl+ \rb*cos(\a)}
\pgfmathsetmacro{\xb}{\xl+ \rbb*cos(\a)}

\pgfmathsetmacro{\ys}{\yl+ \rb*sin(\a)}
\pgfmathsetmacro{\yb}{\yl+ \rbb*sin(\a)}

\draw[fill=black]
(\xs,\ys)--(\xb,\yb);

\fill[black] (\xb,\yb) circle(0.5mm);
}

\pgfmathsetmacro{\xsi}{\xl+ \rb*cos(0)}
\pgfmathsetmacro{\xbi}{\xl+ \rbb*cos(0)}
\pgfmathsetmacro{\xmi}{(\xsi+\xbi)/2}

\pgfmathsetmacro{\ysi}{\yl+ \rb*sin(0)}
\pgfmathsetmacro{\ybi}{\yl+ \rbb*sin(0)}
\pgfmathsetmacro{\ymi}{(\ysi+ \ybi)/2}
\draw[<-] (\xmi,\ymi) -- (\xbi,\ybi);

\pgfmathsetmacro{\xse}{\xl+ \rb*cos(22.5)}
\pgfmathsetmacro{\xbe}{\xl+ \rbb*cos(22.5)}
\pgfmathsetmacro{\xme}{(\xse+\xbe)/2}

\pgfmathsetmacro{\yse}{\yl+ \rb*sin(22.5)}
\pgfmathsetmacro{\ybe}{\yl+ \rbb*sin(22.5)}
\pgfmathsetmacro{\yme}{(\yse+ \ybe)/2}
\draw[->] (\xse,\yse) -- (\xme,\yme);

\pgfmathsetmacro{\xbm}{(\xbi+ \xbe)/2}
\pgfmathsetmacro{\ybm}{(\ybi+ \ybe)/2}

\node (n0) at (\xbm,\ybm) [place] {\tiny$I^{(1)}$};
\draw (\xbi,\ybi) -- (n0)
           (n0) -- (\xbe,\ybe);

\foreach \b in {45,67.5,90,112.5,135,157.5,180,202.5,225,247.5,270,292.5,315}
{
\pgfmathsetmacro{\initialx}{\xl+ \rbb*cos(\b)}
\pgfmathsetmacro{\endx}{\xl+ \rbb*cos(\b+22.5)}
\pgfmathsetmacro{\middlex}{(\initialx+\endx)/2}

\pgfmathsetmacro{\initialy}{\yl+ \rbb*sin(\b)}
\pgfmathsetmacro{\endy}{\yl+ \rbb*sin(\b+22.5)}
\pgfmathsetmacro{\middley}{(\initialy+\endy)/2}

\pgfmathsetmacro{\value}{1+\b/22.5}

\node (n) at (\middlex,\middley) [place] {\fontsize{2}{2.5}\selectfont $U^{(1)}_{\pgfmathprintnumber{\value}}$};

\draw (\initialx,\initialy) -- (n)
           (n) -- (\endx,\endy);
}

\def\xr{2}
\def\yr{4}

\coordinate (Cr) at (\xr,\yr);

\draw (\xr,\yr) circle (\rs);
\node (omegar) at (\xr, \yr) {\Large$\Omega$};
\node (text) at (0, 6) {b)};

\foreach \ar in {0,22.5,45,67.5,90,112.5,135,157.5,180,202.5,225,247.5,270,292.5,315,337.5}
{
\pgfmathsetmacro{\xssr}{\xr+ \rs*cos(\ar-4)}
\pgfmathsetmacro{\xsbr}{\xr+ \rb*cos(\ar-4)}
\pgfmathsetmacro{\xbsr}{\xr+ \rs*cos(\ar+4)}
\pgfmathsetmacro{\xbbr}{\xr+ \rb*cos(\ar+4)}

\pgfmathsetmacro{\yssr}{\yr+ \rs*sin(\ar-4)}
\pgfmathsetmacro{\ysbr}{\yr+ \rb*sin(\ar-4)}
\pgfmathsetmacro{\ybsr}{\yr+ \rs*sin(\ar+4)}
\pgfmathsetmacro{\ybbr}{\yr+ \rb*sin(\ar+4)}

\draw[fill=black]
(\xssr,\yssr)--(\xsbr,\ysbr)--(\xbbr,\ybbr)--(\xbsr,\ybsr)--cycle;
}

\foreach \ar/\artext in {
0/E_1,
22.5/E_2,
45/E_3,
67.5/E_4,
90/E_5,
112.5/E_6,
135/E_7,
157.5/E_8,
180/E_9,
202.5/E_{10},
225/E_{11},
247.5/E_{12},
270/E_{13},
292.5/E_{14},
315/E_{15},
337.5/E_{16}}
{
\coordinate (Tr) at ($(Cr)+(\ar:1.33)$);
\draw (Tr) node {\tiny{$\artext$}};
}

\shade[line width=2pt, top color=black] 
  (1.5, 5) to [out=200, in=120]   (2.8,4.3)
         to [out=20, in=50] (1.5,5) ; 
         
\shade[line width=2pt, top color=black] 
  (1.5, 3) to [out=20, in=10]   (2.5,3.8)
         to [out=200, in=150] (1.5,3) ; 
         
                 
\foreach \a in {0,22.5,45,67.5,90,112.5,135,157.5,180,202.5,225,247.5,270,292.5,315,337.5}
{
\pgfmathsetmacro{\xs}{\xr+ \rb*cos(\a)}
\pgfmathsetmacro{\xb}{\xr+ \rbb*cos(\a)}

\pgfmathsetmacro{\ys}{\yr+ \rb*sin(\a)}
\pgfmathsetmacro{\yb}{\yr+ \rbb*sin(\a)}

\draw[fill=black]
(\xs,\ys)--(\xb,\yb);

\fill[black] (\xb,\yb) circle(0.5mm);
}

\pgfmathsetmacro{\xsi}{\xr+ \rb*cos(22.5)}
\pgfmathsetmacro{\xbi}{\xr+ \rbb*cos(22.5)}
\pgfmathsetmacro{\xmi}{(\xsi+\xbi)/2}

\pgfmathsetmacro{\ysi}{\yr+ \rb*sin(22.5)}
\pgfmathsetmacro{\ybi}{\yr+ \rbb*sin(22.5)}
\pgfmathsetmacro{\ymi}{(\ysi+ \ybi)/2}
\draw[<-] (\xmi,\ymi) -- (\xbi,\ybi);

\pgfmathsetmacro{\xse}{\xr+ \rb*cos(45)}
\pgfmathsetmacro{\xbe}{\xr+ \rbb*cos(45)}
\pgfmathsetmacro{\xme}{(\xse+\xbe)/2}

\pgfmathsetmacro{\yse}{\yr+ \rb*sin(45)}
\pgfmathsetmacro{\ybe}{\yr+ \rbb*sin(45)}
\pgfmathsetmacro{\yme}{(\yse+ \ybe)/2}
\draw[->] (\xse,\yse) -- (\xme,\yme);

\pgfmathsetmacro{\xbm}{(\xbi+ \xbe)/2}
\pgfmathsetmacro{\ybm}{(\ybi+ \ybe)/2}

\node (n0) at (\xbm,\ybm) [place] {\tiny$I^{(2)}$};
\draw (\xbi,\ybi) -- (n0)
           (n0) -- (\xbe,\ybe);

\foreach \b in {67.5,90,112.5,135,157.5,180,202.5,225,247.5,270,292.5,315,337.5}
{
\pgfmathsetmacro{\initialx}{\xr+ \rbb*cos(\b)}
\pgfmathsetmacro{\endx}{\xr+ \rbb*cos(\b+22.5)}
\pgfmathsetmacro{\middlex}{(\initialx+\endx)/2}

\pgfmathsetmacro{\initialy}{\yr+ \rbb*sin(\b)}
\pgfmathsetmacro{\endy}{\yr+ \rbb*sin(\b+22.5)}
\pgfmathsetmacro{\middley}{(\initialy+\endy)/2}

\pgfmathsetmacro{\value}{1+\b/22.5}

\node (n) at (\middlex,\middley) [place] {\fontsize{2}{2.5}\selectfont $U^{(2)}_{\pgfmathprintnumber{\value}}$};

\draw (\initialx,\initialy) -- (n)
           (n) -- (\endx,\endy);
}

\end{tikzpicture}
\caption{The Neighboring Method: a) first current pattern, b) second current pattern.} \label{fig:M1}
\end{figure}
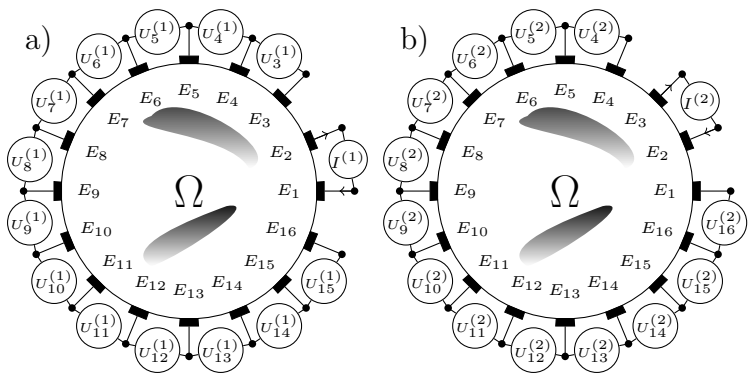


Assuming that the electrodes $E_l$ are relatively open and connected subsets of $\partial \Omega$, that they are perfectly conducting and that contact impedances are negligible, the resulting electric potential $u^{(k)}$ at the $k$-th current pattern obeys the following mathematical model  (the so-called \textit{shunt model} \cite{cheng1989electrode}):
\begin{eqnarray}\label{eqn:shunt}
 \begin{array}{lllll}
& \nabla \cdot (\sigma \nabla u) &=& 0 & \mbox{ in } \Omega, \\
& \int_{E_l} \sigma \partial_{\nu} u \; {\rm d}s &=& I^{(k)}_l & \mbox{ for } l =1,\dots, L,\\
& \sigma \partial_{\nu} u &=&0 & \mbox{ on } \partial \Omega \setminus \bigcup_{l=1}^L E_l,\\
&u|_{E_l} &=& \mbox{const.} & \mbox{ for } l=1,\dots, L.
\end{array}
\end{eqnarray}
Here $\nu$ is the unit normal vector on $\partial\Omega$ pointing outward and $I^{(k)}_l:=(\delta_{k,l}-\delta_{k+1,l})I$ describes the $k$-th applied current pattern where a current of strength $I>0$ is driven through the $k$-th and $(k+1)$-th electrode.
Notice that $\{I^{(k)}_l\}$ satisfy the conservation of charge $\sum_{l=1}^L I^{(k)}_l=0$, and that the electric potential $u^{(k)}$ is uniquely determined by (\ref{eqn:shunt}) only up to the addition of a constant.
The voltage measurements are given by
\begin{eqnarray}\label{eqn:voltage}
U^{(k)}_{l}:=u^{(k)}|_{E_l}-u^{(k)}|_{E_{l+1}}.
\end{eqnarray}

The herein used shunt model ignores the effect of contact impedances between the electrodes and the imaging domain. 
This is only valid when voltages are measured on small (see \cite{hanke2011justification}) and current-free electrodes, so that 
(\ref{eqn:shunt}) correctly models only the measurements $U^{(k)}_{l}$ with $|k-l|>1$. For difference measurements, the missing 
elements $U^{(k)}_l$ with $|k-l|\leq 1$, on the other hand, can be calculated by interpolation taking into account reciprocity, conservation of voltages and the geometry-specific smoothness of difference EIT data, cf. \cite{harrach2015interpolation}.
For an imaging subject with unknown conductivity $\sigma$, one thus obtains a full matrix of measurements
$U(\sigma)=(U^{(k)}_l)_{k,l=1,\ldots,L}$.




\section{Monotonicity-based regularization}
\label{Sec:algorithm}

\subsection{Standard one-step linearization methods}\label{subsec:1}
In difference EIT, the measurements $U(\sigma)$ are compared with measurements $U(\sigma_0)$ for some reference conductivity distribution $\sigma_0$ in order to reconstruct the conductivity difference $\sigma-\sigma_0$. 
This is usually done by a single linearization step 
\[
U'(\sigma_0) (\sigma-\sigma_0) \approx U(\sigma)-U(\sigma_0).
\]
where $U'(\sigma_0):L^{\infty}(\Omega)\to \R^{L\times L}$ is the Fr\'echet derivative of the voltage measurements
\[
U'(\sigma_0):\kappa\mapsto \left(-\int_{\Omega}\kappa \nabla u^{(k)}_{\sigma_0}\cdot \nabla u^{(l)}_{\sigma_0}\,{\rm d}x\right)_{1\le k,l \le L}
\]

We discretize the reference domain $\overline{\Omega}=\cup_{j=1}^P \overline{P}_j$ into $P$ disjoint open pixels $P_j$ and 
make the piecewise-constant Ansatz
\[
\kappa(x)= \sum_{j=1}^P \kappa_j \chi_{P_j}(x).
\]
This approach leads to the linear equation
\begin{equation}\label{eqn:one_step_linearized}
\mathbf{S} \mathbf{\kappa} = \mathbf{V}
\end{equation}
where $\mathbf{V}$ and the columns of the \emph{sensitivity matrix} $\mathbf{S}$ contain the entries of the measurements $U(\sigma)-U(\sigma_0)$ and the discretized Fr\'echet derivative, resp., written as long vectors, i.e.,
\begin{eqnarray*} 
\mathbf{\kappa}=(\kappa_j)_{j=1}^P\in \R^P,\\
\mathbf{V}=(V_{i})_{i=1}^{L^2}\in \R^{L^2}, \quad \mbox{ with } V_{(l-1)L+k}=U_l^{(k)}(\sigma)-U_l^{(k)}(\sigma_0),\\
\mathbf{S}=(S_{i,j})\in \R^{L^2,P}, \quad \mbox{ with } S_{(l-1)L+k, j}=-\int_{P_j} \nabla u^{(k)}_{\sigma_0}\cdot \nabla u^{(l)}_{\sigma_0}\,{\rm d}x.
\end{eqnarray*}

Most practically used EIT algorithms are based on solving a regularized variant of (\ref{eqn:one_step_linearized})
to obtain an approximation $\kappa$ to the conductivity difference $\sigma-\sigma_0$. 
The popular algorithms NOSER \cite{cheney1990noser} and GREIT \cite{adler2009greit} use (generalized) Tikhonov regularization and minimize 
\[
\norm{\mathbf{S} \mathbf{\kappa} -  \mathbf{V} }^2_\mathrm{res} + \alpha \norm{ \mathbf{\kappa} }^2_\mathrm{pen} \to \mbox{min!}
\]
with (heuristically chosen) weighted Euclidian norms $\norm{\cdot}_\mathrm{res}$ and $\norm{\cdot}_\mathrm{pen}$ in the residuum and penalty term.

\subsection{Monotonicity-based regularization}

It has been shown in \cite{harrach2010exact} that shape information in EIT is invariant under linearization.
Thus one-step linearization methods are principally capable of reconstructing the correct (outer) support of the
conductivity difference even though they ignore the non-linearity of the EIT measurement process. 
In \cite{harrach2015enhancing} the authors developed a monotonicity-based regularization method for the linearized EIT equation for which (in the continuum model)
it can be guaranteed that the regularized solutions converge against a function that shows the correct outer shape. In this section, we formulate and analyze this new method for real electrode measurements, and in the next section we will apply it to real data from a phantom experiment and compare it with the GREIT method. 

The main idea of monotonicity-based regularization is to minimize the residual of the linearized equation~(\ref{eqn:one_step_linearized}) 
\[
\norm{\mathbf{S} \mathbf{\kappa} -  \mathbf{V} }^2 \to \mbox{min!}
\] 
with constraints on the entries of $\mathbf{\kappa}$ that are obtained from monotonicity tests. 

For the following, we assume that the background is homogeneous and that all anomalies are more conductive, or all anomalies are less conductive than the background, i.e., $\sigma_0$ is constant, and either 
\[
\sigma(x)=\sigma_0+\gamma(x) \chi_D(x), \quad \mbox{ or } \quad
\sigma(x)=\sigma_0-\gamma(x) \chi_D(x).
\]
$D$ is an open set denoting the conductivity anomalies, and $\gamma:\ D\to \R$ is the contrast of the anomalies. We furthermore assume that we are given a lower bound $c>0$ of the anomaly contrast, i.e. $\gamma(x)\geq c$.

For the monotonicity tests it is crucial to consider the measurements and the columns of the sensitivity matrix $\mathbf{S}$ as matrices and compare
them in terms of matrix definiteness, cf.\ \cite{choi2014regularizing,harrach2015resolution,harrach2015interpolation} for the origins of this sensitivity matrix based approach. Let $V:=U(\sigma)-U(\sigma_0)\in \R^{L\times L}$ denote the EIT difference measurements written as $L\times L$-matrix, and $S_k\in \R^{L\times L}$ denote the $k$-th column of the sensitivity matrix written as $L\times L$-matrix, i.e. the $(j,l)$-th entry of $S_k$ is given by
\[
-\int_{P_k} \nabla u^{(j)}_{\sigma_0}\cdot \nabla u^{(l)}_{\sigma_0}\,{\rm d}x.
\]
We then define for each pixel $P_k$ 
\begin{eqnarray}\label{betak}
\beta_k:=\max\{\alpha \ge 0: \alpha S_k \geq -|V|\},
\end{eqnarray}
where $|V|$ denotes the matrix absolute value of $V$, and the comparison 
$\alpha S_k \geq -|V|$ is to be understood in the sense of matrix definiteness, i.e. 
$\alpha S_k \geq -|V|$ holds if and only if all eigenvalues of $\alpha S_k +|V|$ are non-negative.

Following \cite{harrach2015enhancing} we then solve the linearized EIT equation (\ref{eqn:one_step_linearized}) using
the monotonicity constraints $\beta_k$. We minimize the Euclidean norm of the residuum
\begin{equation}\label{eq:min_lin_res}
\norm{\mathbf{S} \mathbf{\kappa} -  \mathbf{V} }^2 \to \mbox{min!}
\end{equation}
under the constraints that
\begin{enumerate}[(C1)]
\item in the case $\sigma\geq \sigma_0$: $0\leq \kappa_k \leq \min(a_+,\beta_k)$, and
\item in the case $\sigma\leq \sigma_0$: $0\geq \kappa_k \geq -\min(a_-,\beta_k)$.
\end{enumerate}
where $a_+:=\sigma_0-\frac{\sigma_0^2}{\sigma_0+c}$, and $a_-:=c$.

For noisy data $V^\delta$ with $\norm{V^\delta-V}\leq V$ this approach can be regularized by replacing $\beta_k$ 
with
\begin{eqnarray}\label{betakdelta}
\beta_k^\delta:=\max\{\alpha \ge 0: \alpha S_k \geq -|V|-\delta I\},
\end{eqnarray}
where $I\in \R^{L\times L}$ is the identity matrix. For the implementation of $\beta_k^\delta$ see section \ref{Sec:num}.

For the continuum model, and under the assumption that $D$ has connected complement, the authors \cite{harrach2015enhancing} showed that for exact data this monotonicity-constrained minimization of the linearized EIT residuum admits a unique solution and that
the support of the solution agrees with the anomalies support $D$ up to the pixel partition. Moreover, \cite{harrach2015enhancing} also shows 
that for noisy data and using the regularized constraints $\beta_k^\delta$, minimizers exist and that, for $\delta\to 0$, they converge to the
minimizer with the correct support. Since practical electrode measurements can be regarded as an approximation to the continuum model, we therefore expect that the above approach will also well approximate the anomaly support for real electrode data.

In the continuum model, the constraints $\beta_k$ will be zero outside the support of the anomaly and positive for each pixel inside the anomaly. The first property relies
on the existence of localized potentials \cite{gebauer2008localized} and is only true in the limit of infinitely many, infinitely small electrodes. The latter property is however true for any number of electrodes as the following result shows:

\begin{thm}\label{thm:beta_k}
If $P_k\subseteq D$, then 
\begin{enumerate}[(a)]
\item in the case $\sigma\geq \sigma_0$ the constraint $\beta_k$ fulfills $\beta_k\geq a_+>0$, and
\item in the case $\sigma\leq \sigma_0$ the constraint $\beta_k$ fulfills $\beta_k\geq a_->0$.
\end{enumerate}
\end{thm}
\begin{proof}
If $P_k\subseteq D$ and $\sigma\geq \sigma_0$ then
\[
\frac{\sigma_0}{\sigma}(\sigma-\sigma_0)=\sigma_0-\frac{\sigma_0^2}{\sigma}
\geq \left(\sigma_0-\frac{\sigma_0^2}{\sigma_0+c}\right) \chi_{P_k}=a_+ \chi_{P_k},
\]
and if $P_k\subseteq D$ and $\sigma\leq \sigma_0$ then
\[
\sigma_0-\sigma\geq c P_k = a_- P_k.
\]
Hence, it suffices to show that $\alpha S_k \geq - |V|$ holds for all $\alpha>0$ that fulfill 
\begin{enumerate}[(a)]
\item $\alpha \chi_{P_k} \leq \frac{\sigma_0}{\sigma}(\sigma-\sigma_0)$, or
\item $\alpha \chi_{P_k}\leq \sigma_0-\sigma$. 
\end{enumerate}
We use the following monotonicity relation from \cite[Lemma 3.1]{harrachcombining} (see also \cite{kang1997inverse,ikehata1998size} for the origin of this estimate):
For any vector $g=(g_j)_{j=1}^L \in \R^L$ we have that
\begin{equation}\label{ineq:monotonicity}
\int_{\Omega} \frac{\sigma_0}{\sigma}(\sigma_0-\sigma)\left| \nabla u^{(g)}_{\sigma_0}\right|^2 {\rm d}x 
\geq g^{\top} V g
\geq \int_{\Omega}(\sigma_0-\sigma) \left| \nabla u^{(g)}_{\sigma_0}\right|^2 {\rm d}x,
\end{equation}
with $u^{(g)}_{\sigma_0}= \sum_{j=1}^L g_j \nabla u^{(j)}_{\sigma_0}$.

If $\alpha \chi_{P_k} \leq \frac{\sigma_0}{\sigma}(\sigma-\sigma_0)$, then 
\begin{eqnarray*}
0\geq g^{\top} (\alpha S_k) g  = -\int_{P_k} \alpha \left| \nabla u^{(g)}_{\sigma_0}\right|^2
\geq \int_{\Omega} \frac{\sigma_0}{\sigma}(\sigma_0-\sigma) \left| \nabla u^{(g)}_{\sigma_0}\right|^2
\geq g^{\top} V g,
\end{eqnarray*}
which shows that $|V|=-V\geq - \alpha S_k$.

If $\alpha \chi_{P_k}\leq \sigma_0-\sigma$, then 
\begin{eqnarray*}
0\leq g^{\top} (-\alpha S_k) g  = \int_{P_k} \alpha \left| \nabla u^{(g)}_{\sigma_0}\right|^2
\leq \int_{\Omega} (\sigma_0-\sigma) \left| \nabla u^{(g)}_{\sigma_0}\right|^2
\leq g^{\top} V g,
\end{eqnarray*}
which shows that $|V|=V\geq - \alpha S_k$.
\end{proof}

\section{Numerical results}\label{Sec:num}
In this section, we will test our algorithm on the data set \texttt{iirc}$\_$\texttt{data}$\_$\texttt{2006} measured
by Professor Eung Je Woo's EIT research group in Korea \cite{wi2014multi,oh2007calibration,oh2007multi,oh2011fully}. \texttt{iirc} stands for Impedance Imaging Research Center. The data set \texttt{iirc}$\_$\texttt{data}$\_$\texttt{2006} is publicly available as part of the
open source software framework EIDORS \cite{adler2006uses} (Electrical Impedance and Diffused Optical Reconstruction Software). Since
\texttt{iirc}$\_$\texttt{data}$\_$\texttt{2006} is also frequently used in the EIDORS tutorials, we believe that this is a good benchmark example to test our new algorithm. 

\subsection{Experiment setting}
The data set \texttt{iirc}$\_$\texttt{data}$\_$\texttt{2006}  was collected using the 16-electrode EIT system KHU Mark1 (see \cite{Oh07a} for more information of this system). The reference object was a Plexiglas tank filled with saline. The tank was a cylinder of diameter 0.2m with 0.01m diameter round electrodes attached on its boundary. Saline was filled to about 0.06m depth. Inside the tank, one put a Plexiglas rod of diameter 0.02m. The conductivity of the saline was 0.15 S/m and the Plexiglas rod was basically non-conductive. 
Data acquisition protocol was adjacent stimulation, adjacent measurement with data acquired on all electrodes. 

The data set \texttt{iirc}$\_$\texttt{data}$\_$\texttt{2006} contains the voltage measurements for both homogeneous and non-homogeneous cases. Measurements for the homogeneous case were obtained when the Plexiglas rod was taken away (reference conductivity in this case is 0.15 S/m). In the non-homogeneous case, $100$ different voltage measurements were measured corresponding to $100$ different positions of the Plexiglas rod. 

\subsection{Numerical implementation}

 EIDORS \cite{adler2006uses} (Electrical Impedance and Diffused Optical Reconstruction Software)  is an open source software that is widely used to reconstruct images  in electrical
 impedance tomography and diffuse optical tomography. To reconstruct images with EIDORS, one first needs to build an EIDORS model that fits with the measured data.   In this paper, we shall use the same EIDORS model described in the EIDORS tutorial web-page:
  
\fl \quad \quad \quad \quad \, \verb"http://eidors3d.sourceforge.net/tutorial/EIDORS_basics/tutorial110.shtml"

Figure \ref{fig:3} shows the reconstructed images of the $9$th-inhomogeneous measurements with different regularization parameters using the EIDORS built-in command \texttt{inv$\_$solve}, which follows the algorithm proposed in \cite{adler1996electrical}. We emphasize that, Figure \ref{fig:3b} (regularization parameter is chosen as $0.03$ by default) was considered at the EIDORS tutorial web-page, we show them here again in order to easily compare them with the reconstructed images using our new method later on.

\begin{figure}[htp]
	\begin{center}
		\subfigure[Parameter $= 0.003$]{\label{fig:3a}\includegraphics[scale=0.5]{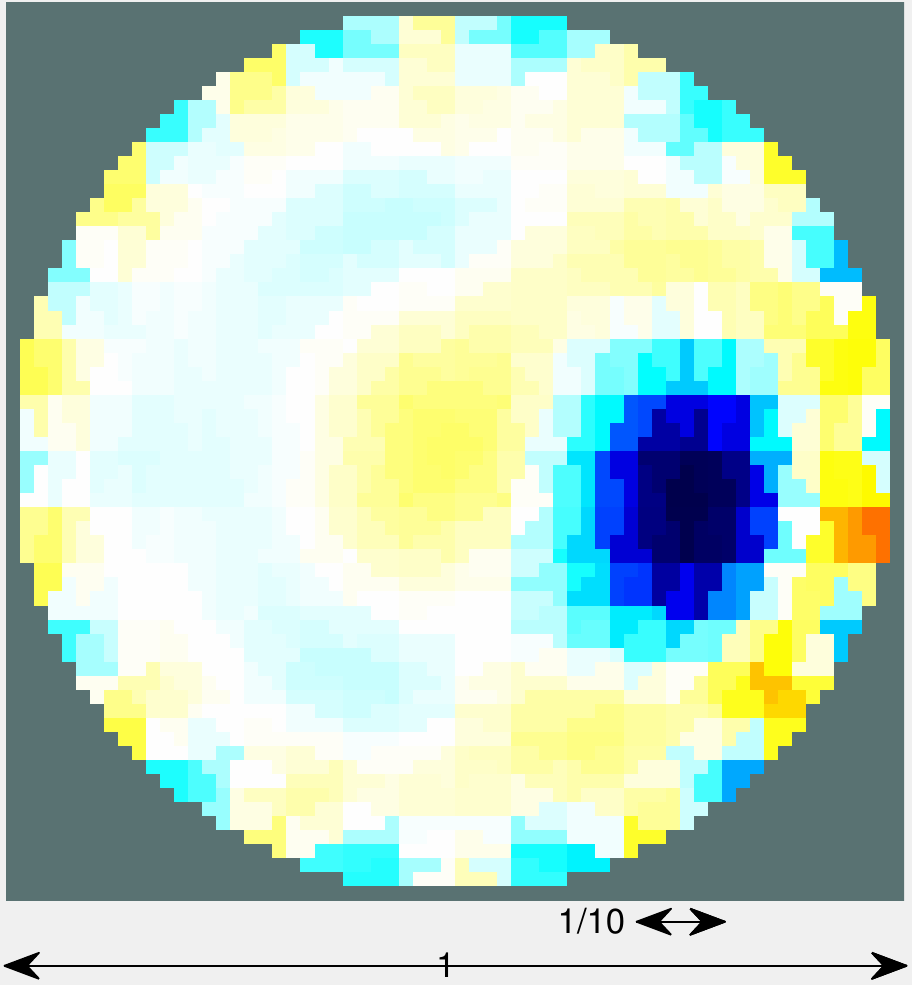}}
		\subfigure[Parameter $= 0.03$]{\label{fig:3b}\includegraphics[scale=0.5]{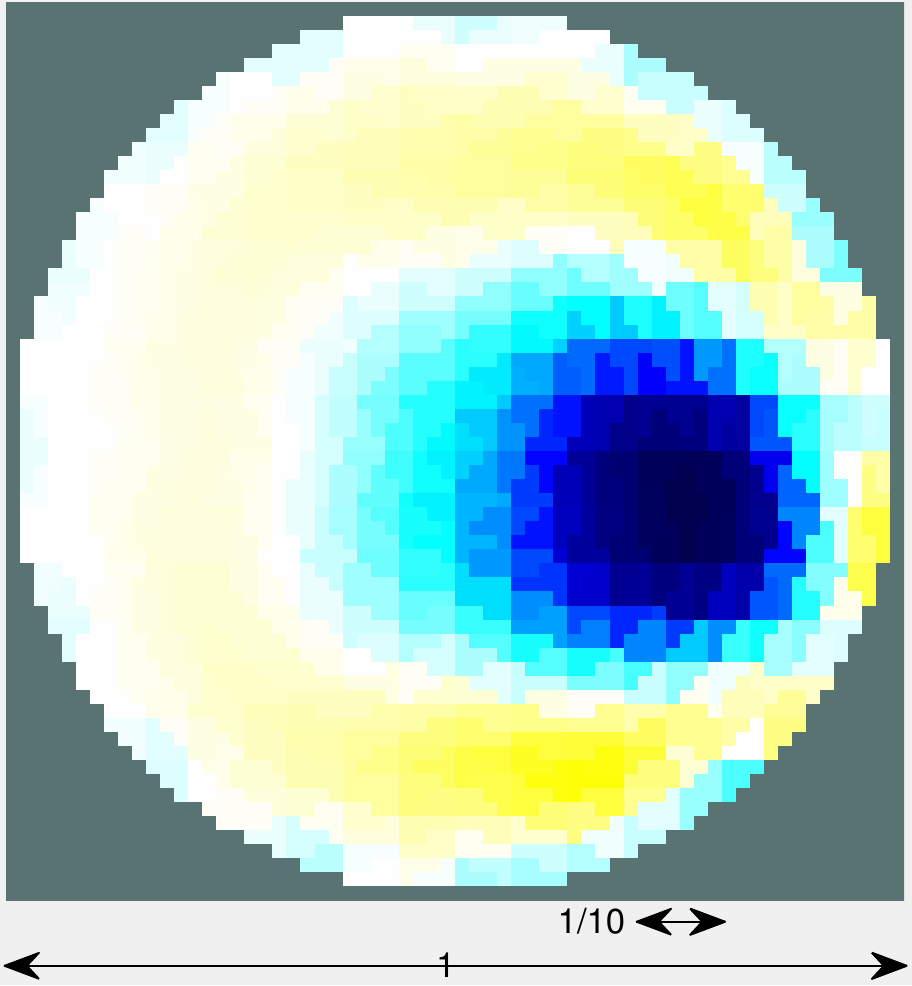}}
		\subfigure[Parameter $= 0.3$]{\label{fig:3c}\includegraphics[scale=0.5]{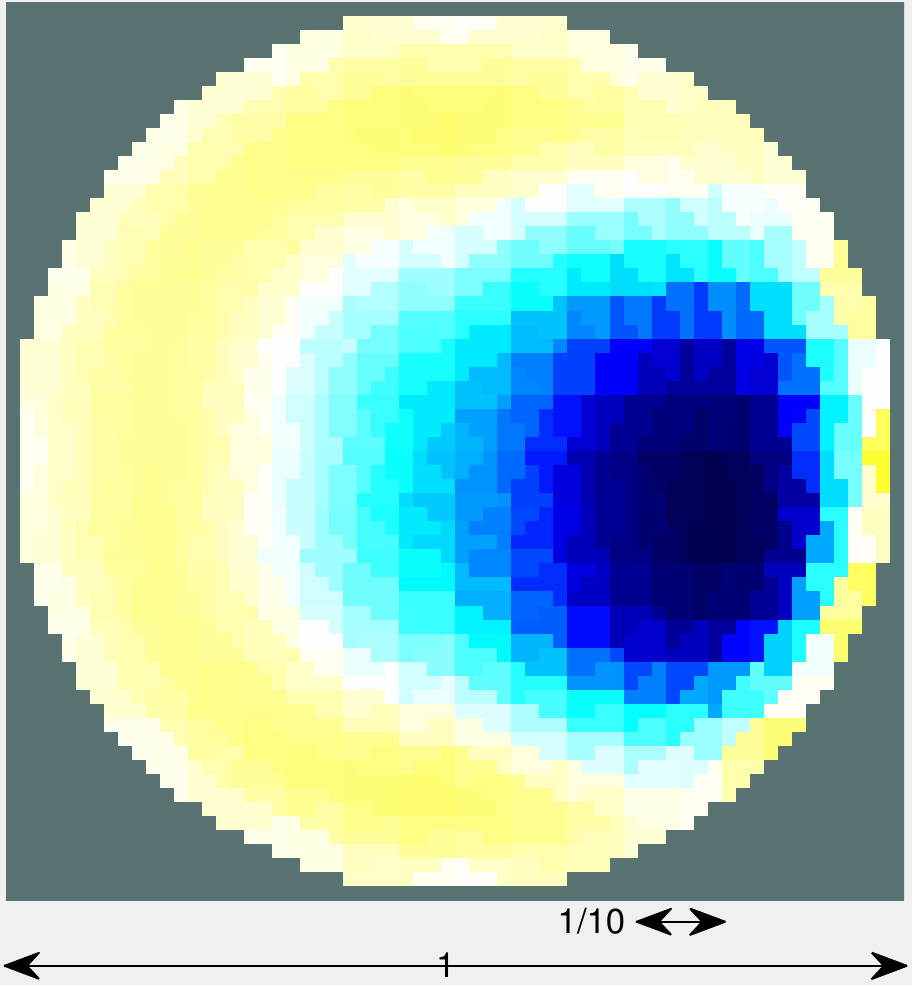}}
	\end{center}
	\caption{Reconstructed images for the $9$th-inhomogeneous voltage measurements with different regularization parameters.}
	\label{fig:3}
\end{figure}

\subsection{Minimizing the residuum}

In the EIDORS model suggested in the EIDORS tutorial web-page, the reference body was chosen by default as a disk of diameter $1$m and the default reference conductivity was $1$ S/m. However, in the experiment setting, the reference body was a cylinder of diameter $0.2$m and the reference conductivity was $0.15$ S/m. Hence, an appropriate scaling factor should be applied to the measurements, to make sure that the EIDORS model fits with these measurements. In the EIDORS tutorial web-page, the measurements were scaled by multiplying by a factor $10^{-4}$. In this paper, to increase the precision of the model, we shall find the best scaling factor that minimizes the error between the measured data and the data generated by the EIDORS model. More precisely, let call \texttt{vh} the measured data for homogeneous case and \texttt{vh$\_$model} the homogeneous data generated by the EIDORS model, the best scaling factor is a minimizer of the following problem
\[
\min_{c \in \R} \|c*\texttt{vh} - \texttt{vh$\_$model}\|_2
\]   
For this experiment setting, the best factor is $2.49577*10^{-5}$. From now on, by measured data we always refer to scaled measured data with respect to this best factor.


\begin{figure}[htp]
	\begin{center}
		\subfigure[\texttt{cvx}]{\label{fig:4a}\includegraphics[scale=0.5]{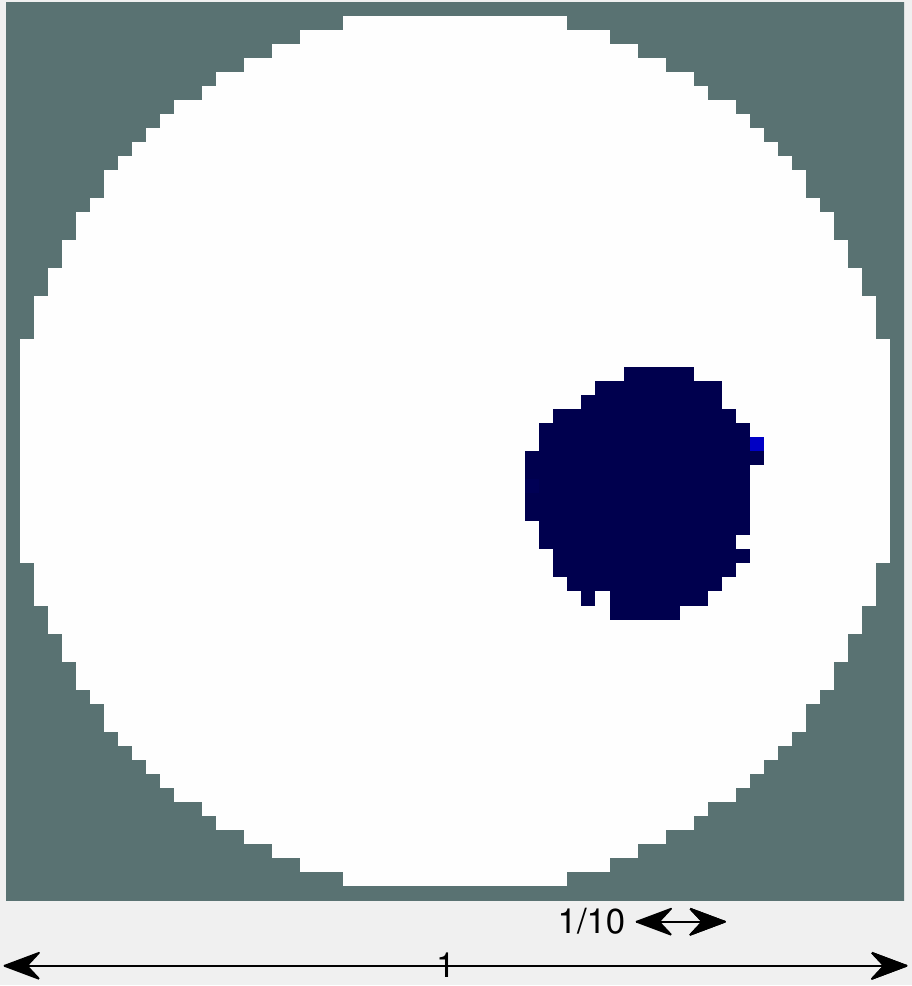}}
		\subfigure[\texttt{quadprog}]{\label{fig:4b}\includegraphics[scale=0.5]{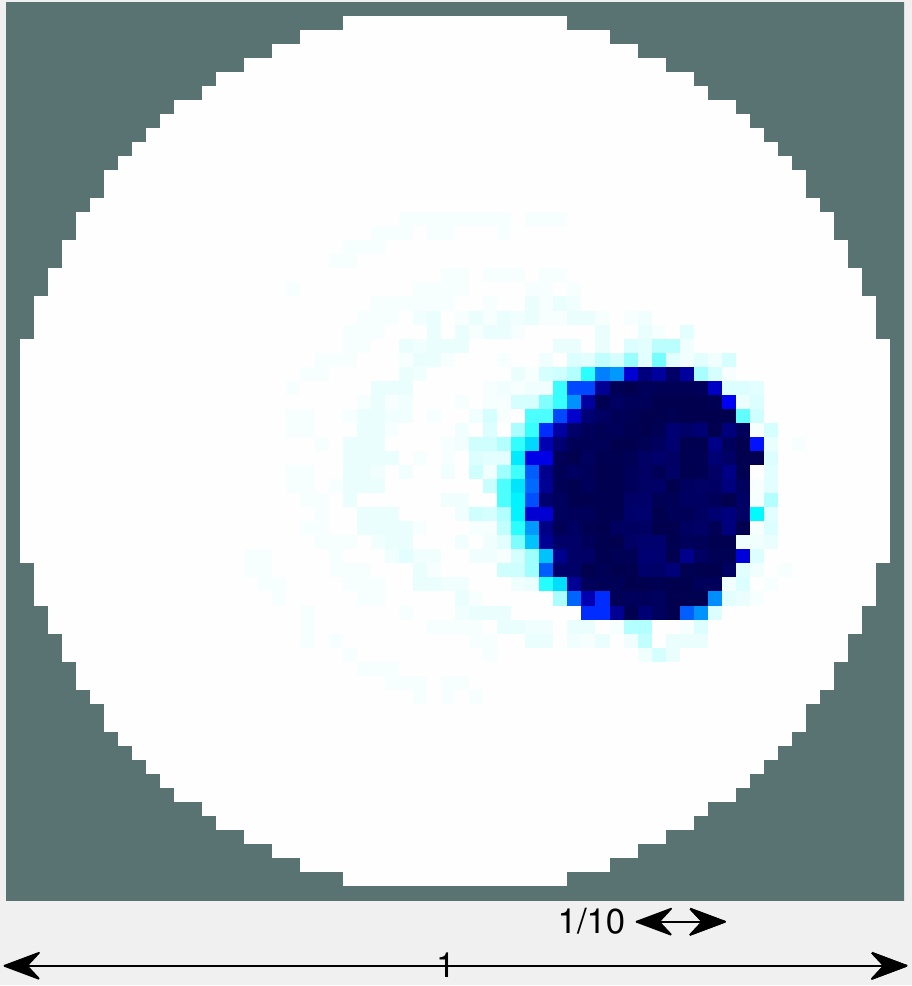}}\\
		\subfigure[EIDORS  \texttt{inv}$\_$\texttt{solve}]{\label{fig:4c}\includegraphics[scale=0.5]{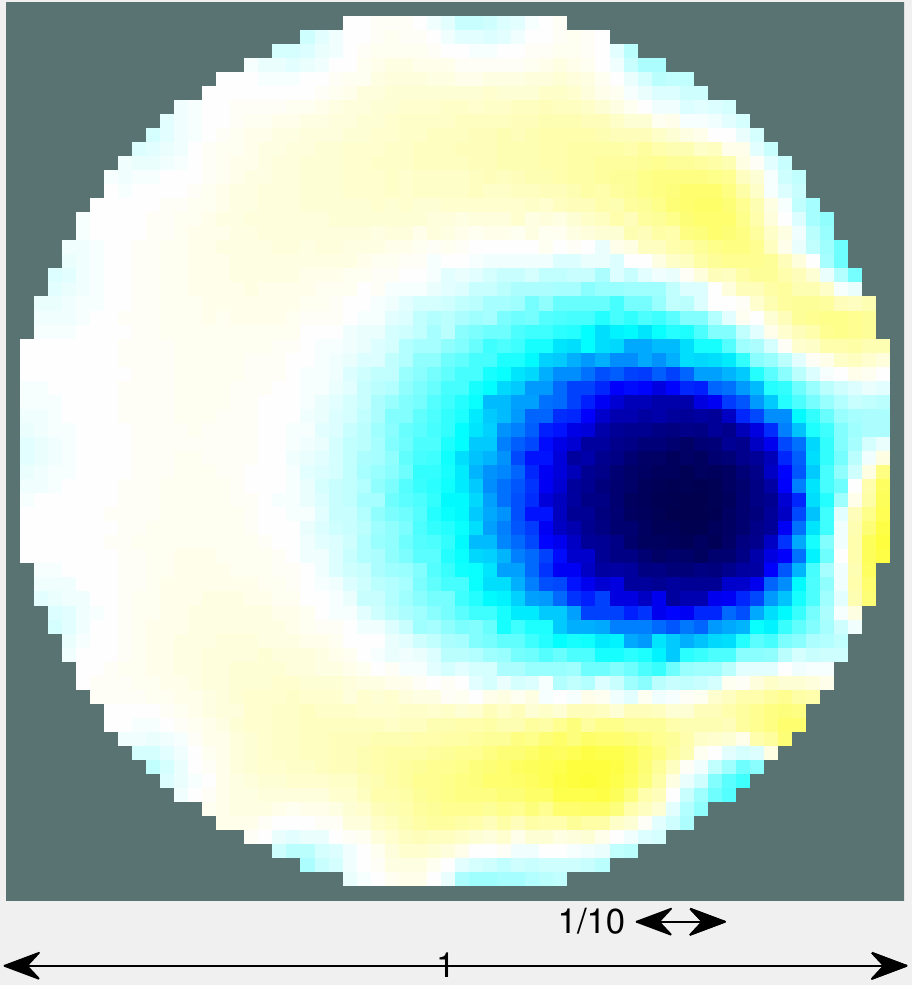}}
		\subfigure[GREIT]{\label{fig:4d}\includegraphics[scale=0.5]{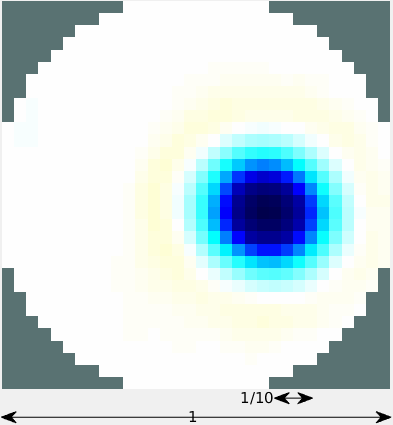}}
	\end{center}
	\caption{Reconstructed images for the $9$th-inhomogeneous voltage measurements with different algorithms (after scaling the measured data w.r.t the best scaling factor).}
	\label{fig:4}
\end{figure}

The next step is to recover the missing measurements on the driving electrodes. We shall follow the result in \cite{harrach2015interpolation} to obtain an approximation for these missing measurements using interpolation.

Now we are in a position to minimize the problem (\ref{eq:min_lin_res}) under the linear constraint (C1) or (C2). To do this, we need to clarify ${\it a_+, a_-},$ and $\beta_k$ in the linear constraints. After scaling, the reference conductivity is $\sigma_0=1$ S/m, and $D$ still denotes the Plexiglas rod with conductivity $\sigma=0$ S/m. Thus, $\gamma=1$, ${\it a}_-=\inf_D \gamma =1$ and $\beta_k$ is calculated using (\ref{betak}). In practice, there is no way to obtain the exact value of the matrix $V$ in (\ref{betak}). Indeed, what we know is just the measured data $V^\delta=U^\delta(\sigma)-U^\delta(\sigma_0)$, where $\delta$ denotes the noise level. When replacing $|V|$ by the noisy version $|V^\delta|$, it may happen that there is no $\alpha>0$ so that the matrix $|V^\delta|+ \alpha S_k$ is still positive semi-definite. Therefore, instead of using (\ref{betak}), we shall calculate $\beta_k$ from
\[\beta_k=\max\{\alpha\ge0\;:\;|V^\delta|+\alpha S_k \ge -\delta I\}.\]
Here, $I$ represents the identity matrix, and $\delta$ is chosen as the absolute value of the smallest eigenvalue of $V^\delta$. Notice that, in the presence of noise, $|V^\delta|+\delta I$ plays the role of the positive semi-definite matrix $|V|$. We shall follow the argument in \cite{harrach2015enhancing} to calculate $\beta_k$. Let $L$ be the lower triangular Cholesky decomposition matrix of $|V^\delta|+\delta I$, and let $\lambda_s(L^{-1}S_k(L^*)^{-1})$ be the smallest eigenvalue of the matrix $L^{-1}S_k(L^*)^{-1}$. Since $S_k$ is negative semi-definite, so is $L^{-1}S_k(L^*)^{-1}$. Thus, $\lambda_s(L^{-1}S_k(L^*)^{-1})\le 0$. Arguing in the same manner as in \cite{harrach2015enhancing}, we get
\[\beta_k=-\frac{1}{\lambda_s(L^{-1}S_k(L^*)^{-1})} \ge 0.\]

The minimizer of (\ref{eq:min_lin_res}) is then obtained using two different approaches: one employs \texttt{cvx} (Figure \ref{fig:4a}), a package for specifying and solving convex programs \cite{cvx,gb08}, the other (Figure \ref{fig:4b}) uses the \texttt{MATLAB}  built-in function \texttt{quadprog} (\texttt{trust-region-reflective} Algorithm). We also show the reconstructed result using the built-in function \texttt{inv$\_$solve} of EIDORS \cite{adler1996electrical} (Figure \ref{fig:4c}) with the default regularization parameter $0.03$ and with GREIT algorithm \cite{adler2009greit} (Figure \ref{fig:4d}) to see that scaling the measured data with the best scaling factor will improve a little bit the reconstructed image. Notice that reconstructed images are highly affected by the choice of the minimization algorithms, and we will see from Figure \ref{fig:4} that the images obtained by \texttt{cvx} has less artifacts than the others.

It is worth to emphasize that although each EIDORS model is assigned to a default regularization parameter, when using the EIDORS built-in function \texttt{inv$\_$solve} \cite{adler1996electrical}, in order to obtain a good reconstruction (Figure \ref{fig:3}) one has to manually choose a regularization parameter, whilst the regularization parameters ${\it a_-}$ and $\beta_k$ in our method are known a-priori provided the information of the conductivity $\sigma$ and the reference conductivity $\sigma_0$ exists. Besides, if we manually choose the parameters $\min({\it a_-,\beta_k})$, we even get much better reconstructed images (Figure \ref{fig:5}).

\begin{figure}[htp]
	\begin{center}
		\subfigure[\texttt{$\min(2,\beta_k)$}]{\label{fig:5a}\includegraphics[scale=0.5]{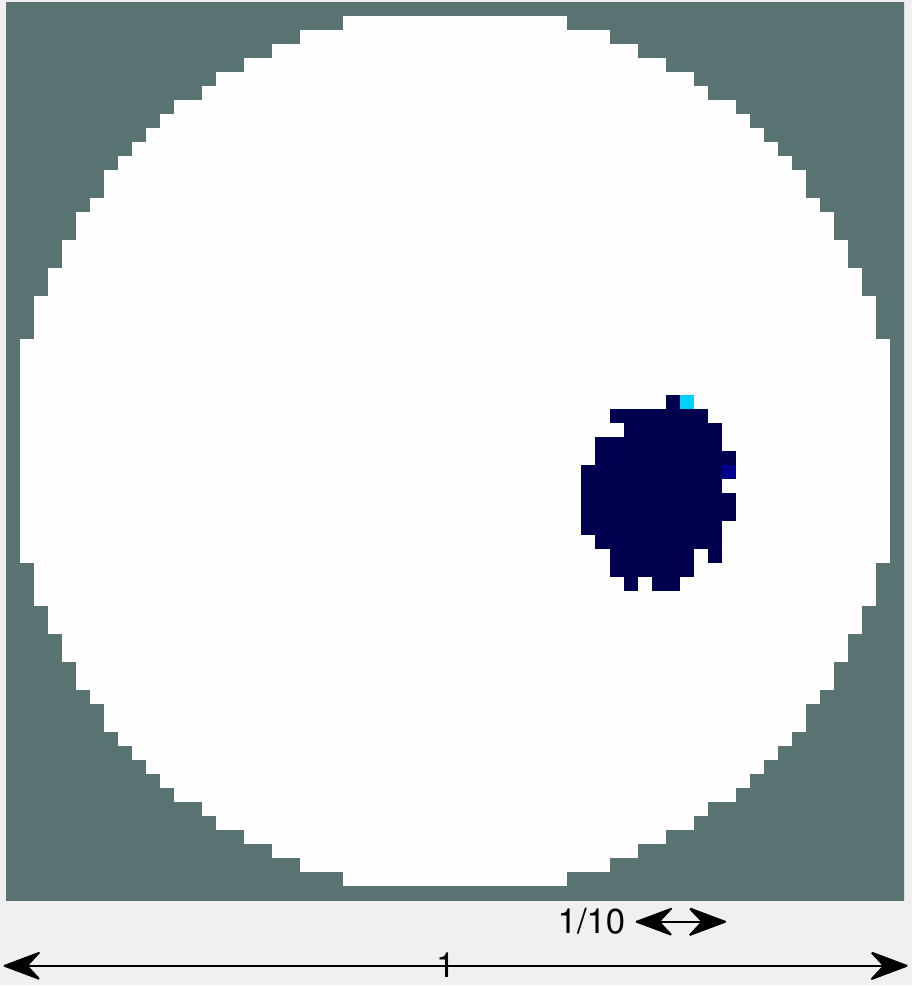}}
		\subfigure[\texttt{$\min(3,\beta_k)$}]{\label{fig:5b}\includegraphics[scale=0.5]{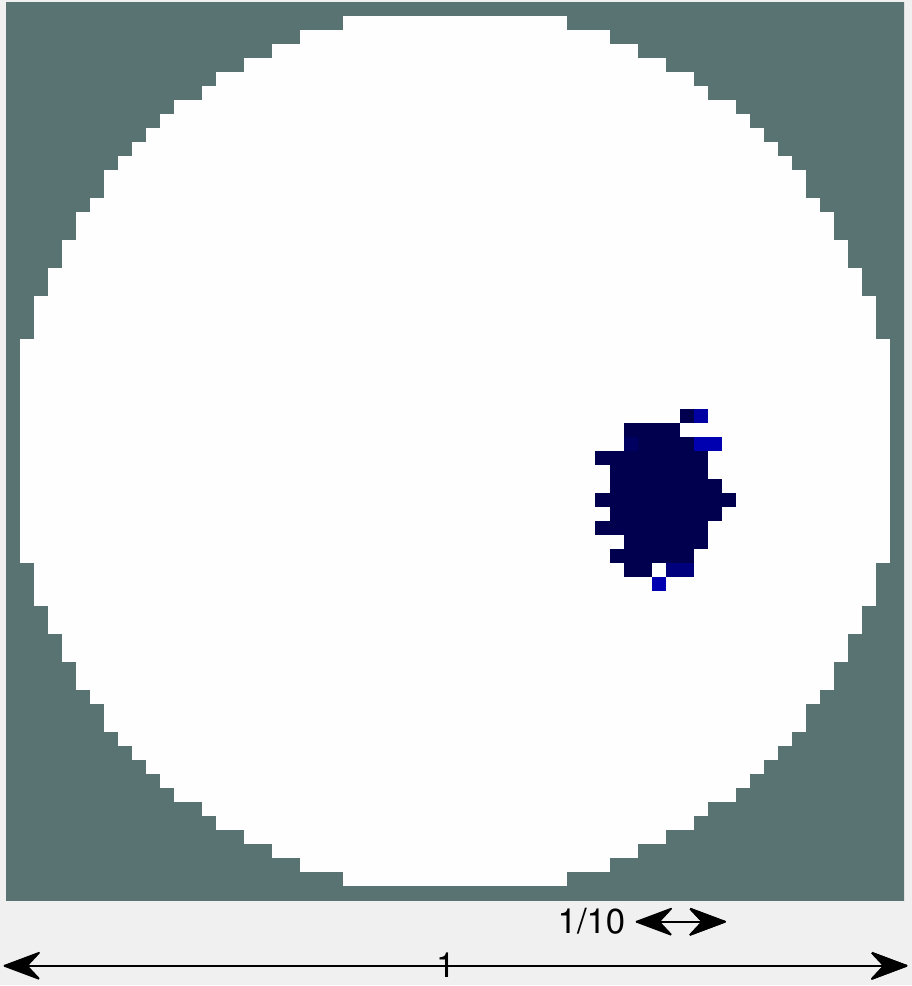}}
		\subfigure[\texttt{$\min(4,\beta_k)$}]{\label{fig:5c}\includegraphics[scale=0.5]{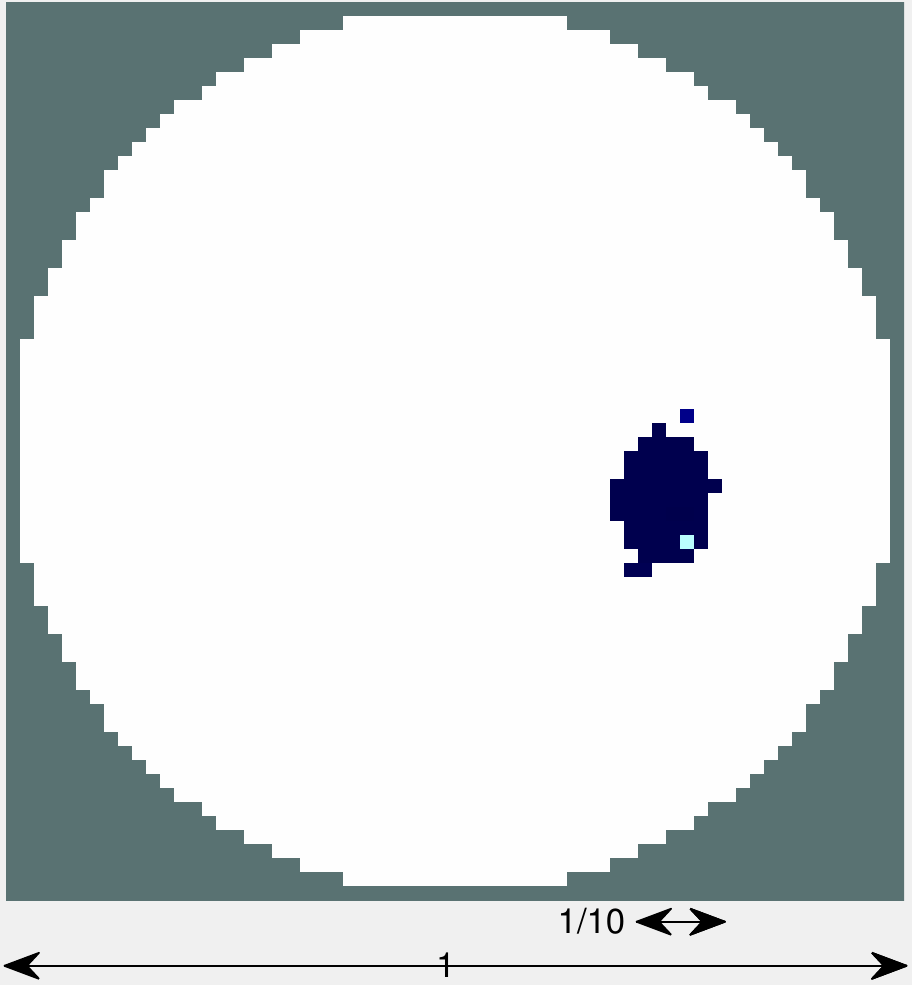}}
	\end{center}
	\caption{Reconstructed images for the $9$th-inhomogeneous voltage measurements with monotonicity-based algorithm and different choices of lower constraint.}
	\label{fig:5}
\end{figure}

\begin{table}
	\caption{\label{tab:1} Runtime of pictures in Figure \ref{fig:4}}
	\begin{tabular*}{\textwidth}{@{}l*{15}{@{\extracolsep{0pt plus
						12pt}}l}}
		\br
		Algorithm  & Runtime (second) \\
		\mr
		\texttt{cvx}  & 839.3892 \\
		\texttt{quadprog (trust-region-reflective)}   &  5.4467\\
		EIDORS (\texttt{inv$\_$solve})   & 0.0231 \\
		GREIT & 0.0120 \\
		\br
	\end{tabular*}
\end{table}


Last but not least, our new method proves its advantage when there are more than one inclusions (Figure \ref{fig:6}). 

\begin{figure}[htp]
	\begin{center}
		\subfigure[]{\label{fig:6a}\includegraphics[scale=0.38]{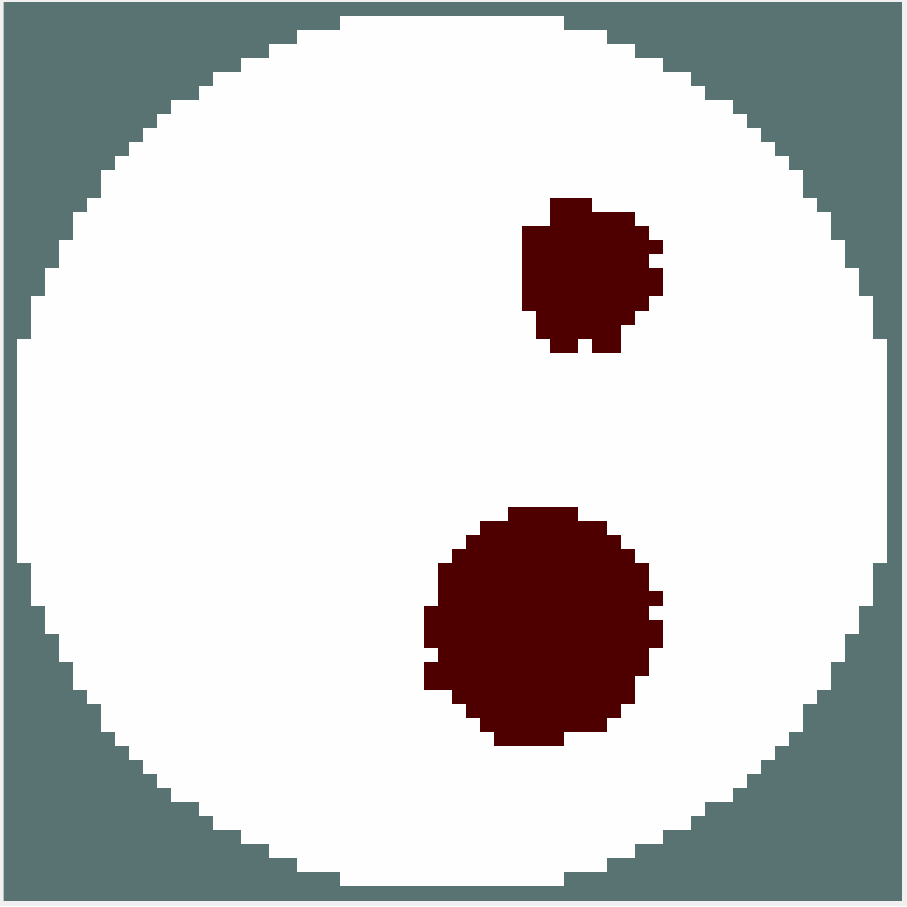}}
		\subfigure[]{\label{fig:6b}\includegraphics[scale=0.38]{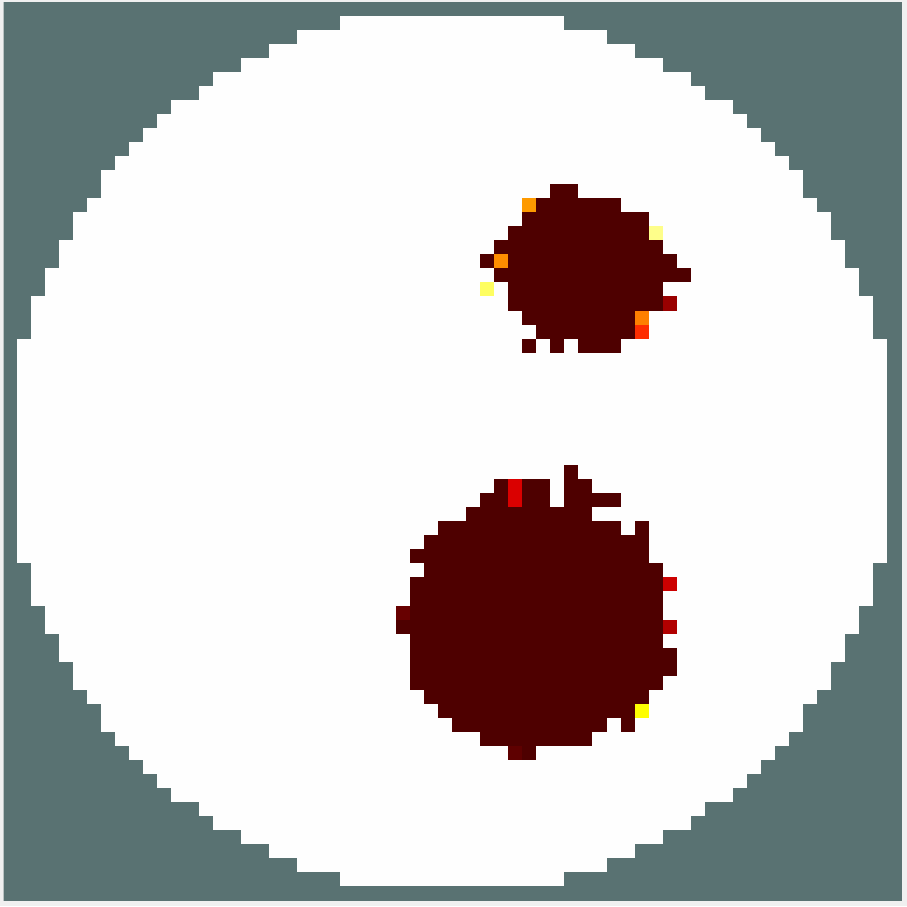}}
		\subfigure[]{\label{fig:6c}\includegraphics[scale=0.38]{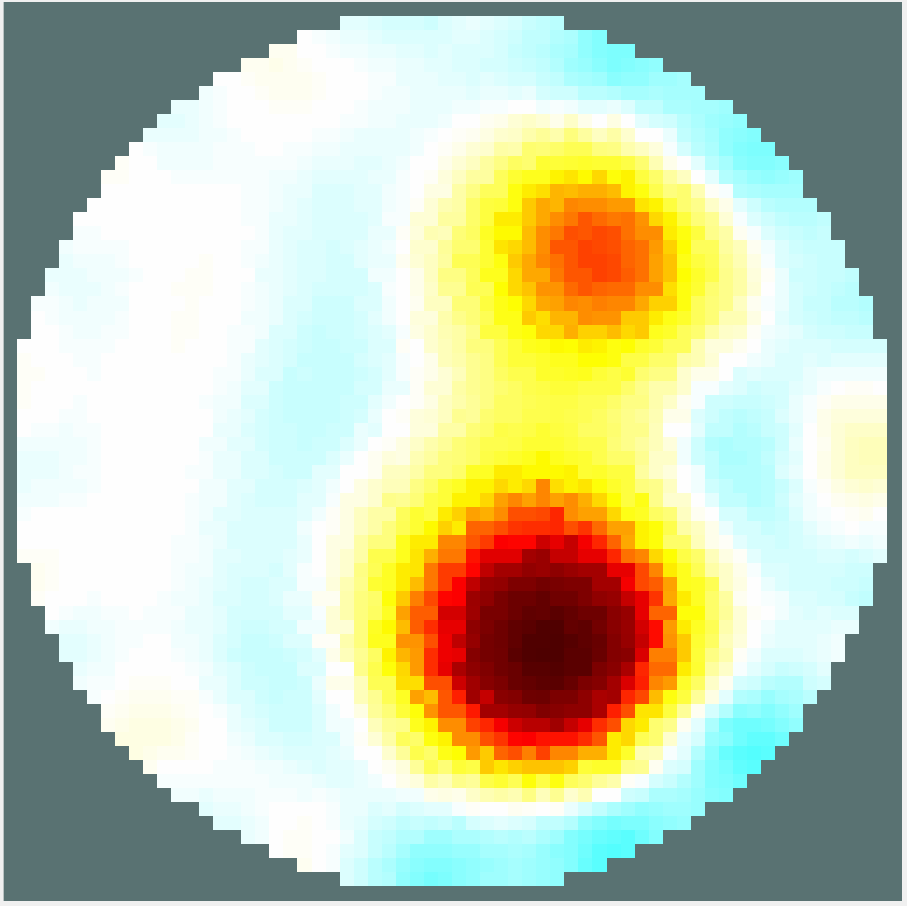}}
		\subfigure[]{\label{fig:6d}\includegraphics[scale=0.38]{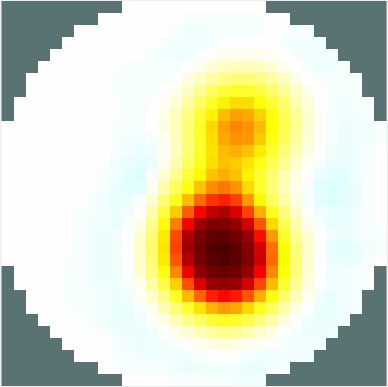}}\\
		\subfigure[]{\label{fig:6a}\includegraphics[scale=0.38]{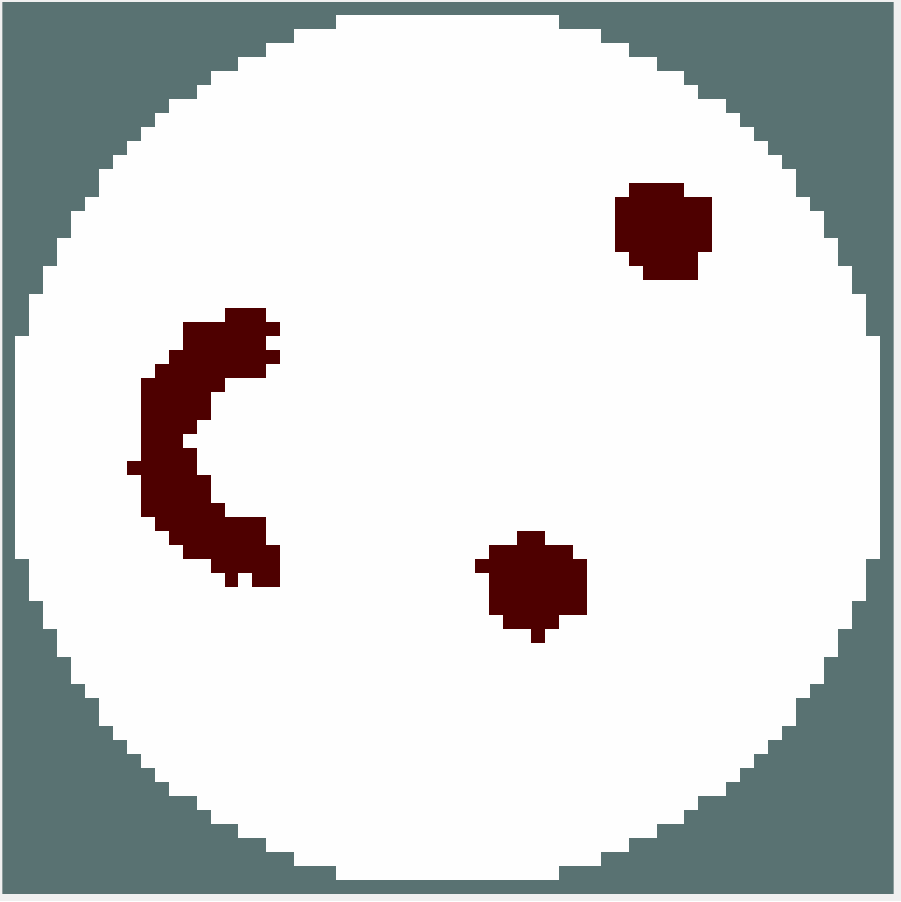}}
			\subfigure[]{\label{fig:6b}\includegraphics[scale=0.38]{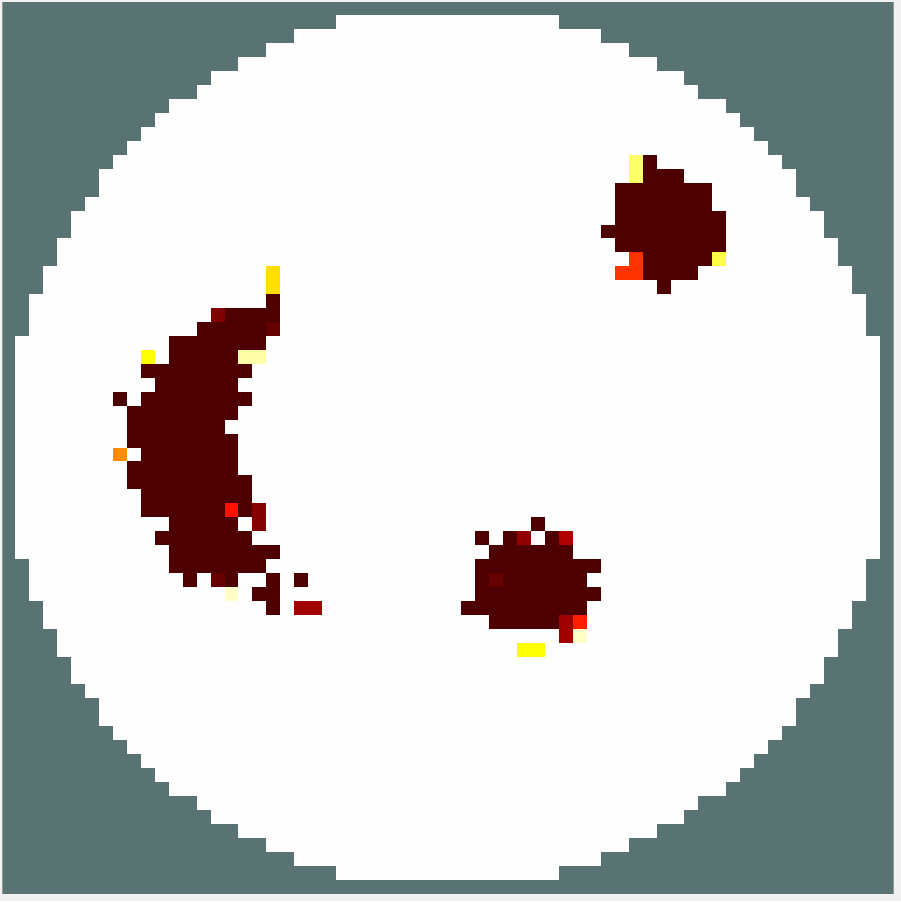}}
			\subfigure[]{\label{fig:6c}\includegraphics[scale=0.38]{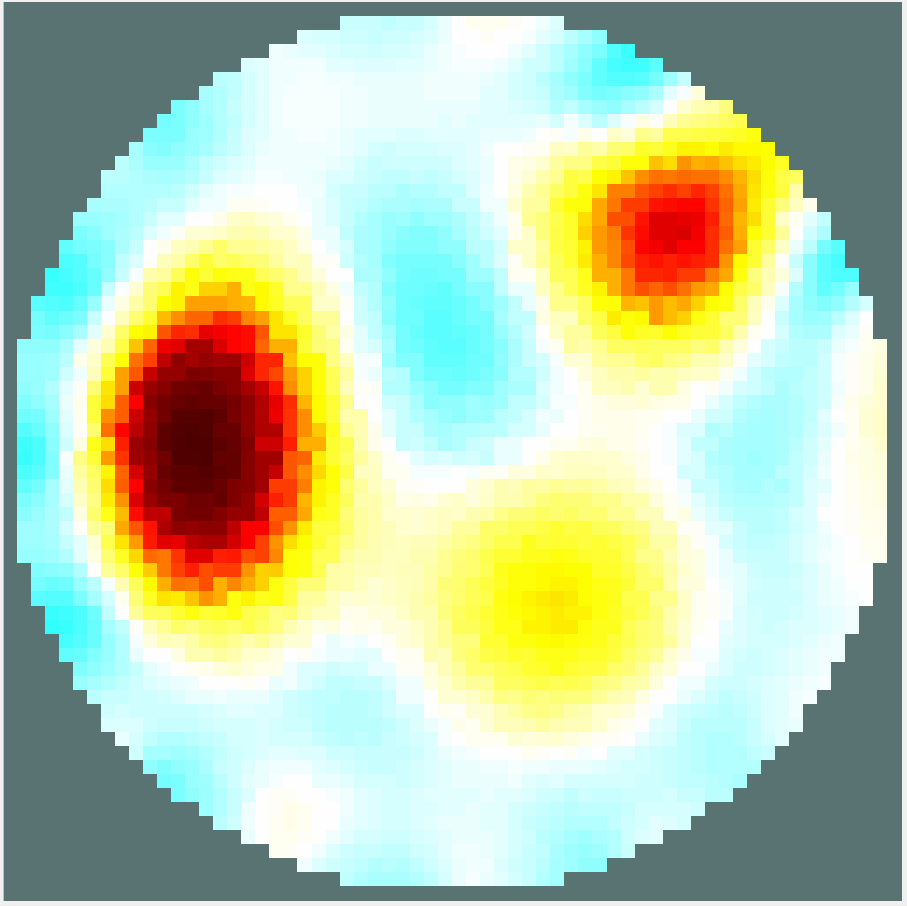}}
			\subfigure[]{\label{fig:6d}\includegraphics[scale=0.38]{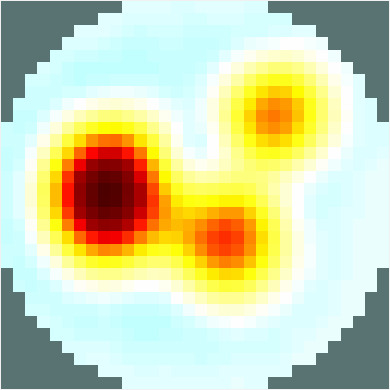}}
	\end{center}
	\caption{Reconstructed images for simulated data with $0.1\%$ noise. (From left to right) First column: True conductivity change, Second column: our new method (with \texttt{cvx}), Third column: EIDORS (\texttt{inv}$\_$\texttt{solve}), Last column: GREIT}
	\label{fig:6}
\end{figure}


\section{Conclusions}\label{Sec:con}
In this paper, we have presented a new algorithm to reconstruct images in EIT in the real electrode setting. Numerical results show that this new algorithm helps to reduce the ringing artifacts in the reconstructed images. Global convergence result of this algorithm has been proved in \cite{harrach2015enhancing} for the Continuum Model. In future works, we shall prove global convergence result for the Shunt Model setting as well as reduce the runtime to fit with real-time applications.


\ack The authors thank Professor Eung Je Woo's EIT research group for the \texttt{iirc} phantom data set. MNM thanks the Academy of Finland (Finnish Center of Excellence in Inverse Problems Research) for financial support of the project number 273979. During part of the preparation of this work, MNM worked at the Department of Mathematics of the Goethe University Frankfurt, Germany.


\section*{References}

\begin{thebibliography}{10}

\bibitem{adler2009greit}
A.~Adler, J.~H. Arnold, R.~Bayford, A.~Borsic, B.~Brown, P.~Dixon, T.~J. Faes,
  I.~Frerichs, H.~Gagnon, Y.~G{\"a}rber, et~al.
\newblock {GREIT}: a unified approach to 2{D} linear {EIT} reconstruction of
  lung images.
\newblock {\em Physiological measurement}, 30(6):S35, 2009.

\bibitem{adler1996electrical}
A.~Adler and R.~Guardo.
\newblock Electrical impedance tomography: regularized imaging and contrast
  detection.
\newblock {\em IEEE transactions on medical imaging}, 15(2):170--179, 1996.

\bibitem{adler2006uses}
A.~Adler and W.~R. Lionheart.
\newblock Uses and abuses of eidors: an extensible software base for eit.
\newblock {\em Physiological measurement}, 27(5):S25, 2006.

\bibitem{aykroyd2006conditional}
R.~G. Aykroyd, M.~Soleimani, and W.~R. Lionheart.
\newblock Conditional {B}ayes reconstruction for {ERT} data using resistance
  monotonicity information.
\newblock {\em Measurement Science and Technology}, 17(9):2405, 2006.

\bibitem{azzouz2007factorization}
M.~Azzouz, M.~Hanke, C.~Oesterlein, and K.~Schilcher.
\newblock The factorization method for electrical impedance tomography data
  from a new planar device.
\newblock {\em International Journal of Biomedical Imaging}, 2007.

\bibitem{brown1987sheffield}
B.~Brown and A.~Seagar.
\newblock The {S}heffield data collection system.
\newblock {\em Clinical Physics and Physiological Measurement}, 8(4A):91, 1987.

\bibitem{cheney1990noser}
M.~Cheney, D.~Isaacson, J.~Newell, S.~Simske, and J.~Goble.
\newblock {NOSER}: {A}n algorithm for solving the inverse conductivity problem.
\newblock {\em International Journal of Imaging Systems and Technology},
  2(2):66--75, 1990.

\bibitem{cheng1989electrode}
K.-S. Cheng, D.~Isaacson, J.~Newell, and D.~G. Gisser.
\newblock Electrode models for electric current computed tomography.
\newblock {\em Biomedical Engineering, IEEE Transactions on}, 36(9):918--924,
  1989.

\bibitem{choi2014regularizing}
M.~K. Choi, B.~Harrach, and J.~K. Seo.
\newblock Regularizing a linearized eit reconstruction method using a
  sensitivity-based factorization method.
\newblock {\em Inverse Problems in Science and Engineering}, 22(7):1029--1044,
  2014.

\bibitem{garde2015convergence}
H.~Garde and S.~Staboulis.
\newblock Convergence and regularization for monotonicity-based shape
  reconstruction in electrical impedance tomography.
\newblock {\em arXiv preprint arXiv:1512.01718}, 2015.

\bibitem{gebauer2008localized}
B.~Gebauer.
\newblock Localized potentials in electrical impedance tomography.
\newblock {\em Inverse Probl. Imaging}, 2(2):251--269, 2008.

\bibitem{gb08}
M.~Grant and S.~Boyd.
\newblock Graph implementations for nonsmooth convex programs.
\newblock In V.~Blondel, S.~Boyd, and H.~Kimura, editors, {\em Recent Advances
  in Learning and Control}, Lecture Notes in Control and Information Sciences,
  pages 95--110. Springer-Verlag Limited, 2008.

\bibitem{cvx}
M.~Grant and S.~Boyd.
\newblock {CVX}: Matlab software for disciplined convex programming, version
  2.1.
\newblock {http://cvxr.com/cvx}, Mar. 2014.

\bibitem{hanke2011justification}
M.~Hanke, B.~Harrach, and N.~Hyv{\"o}nen.
\newblock Justification of point electrode models in electrical impedance
  tomography.
\newblock {\em Mathematical Models and Methods in Applied Sciences},
  21(06):1395--1413, 2011.

\bibitem{scherzer2011handbook}
M.~Hanke and A.~Kirsch.
\newblock Sampling methods.
\newblock In O.~Scherzer, editor, {\em Handbook of Mathematical Models in
  Imaging}, pages 501--550. Springer, 2011.

\bibitem{harrach2013recent}
B.~Harrach.
\newblock Recent progress on the factorization method for electrical impedance
  tomography.
\newblock {\em Computational and mathematical methods in medicine}, 2013.

\bibitem{harrach2015interpolation}
B.~Harrach.
\newblock Interpolation of missing electrode data in electrical impedance
  tomography.
\newblock {\em Inverse Problems}, 31(11):115008, 2015.

\bibitem{harrachcombining}
B.~Harrach, E.~Lee, and M.~Ullrich.
\newblock Combining frequency-difference and ultrasound modulated electrical
  impedance tomography.
\newblock {\em Inverse Problems}, 31(9):095003, 2015.

\bibitem{harrach2015enhancing}
B.~Harrach and M.~N. Minh.
\newblock Enhancing residual-based techniques with shape reconstruction
  features in {E}lectrical {I}mpedance {T}omography.
\newblock {\em arXiv preprint arXiv:1511.07079}, 2015.

\bibitem{harrach2010exact}
B.~Harrach and J.~K. Seo.
\newblock Exact shape-reconstruction by one-step linearization in electrical
  impedance tomography.
\newblock {\em SIAM Journal on Mathematical Analysis}, 42(4):1505--1518, 2010.

\bibitem{harrach2010factorization}
B.~Harrach, J.~K. Seo, and E.~J. Woo.
\newblock Factorization method and its physical justification in
  frequency-difference electrical impedance tomography.
\newblock {\em Medical Imaging, IEEE Transactions on}, 29(11):1918--1926, 2010.

\bibitem{harrach2013monotonicity}
B.~Harrach and M.~Ullrich.
\newblock Monotonicity-based shape reconstruction in electrical impedance
  tomography.
\newblock {\em SIAM Journal on Mathematical Analysis}, 45(6):3382--3403, 2013.

\bibitem{harrach2015resolution}
B.~Harrach and M.~Ullrich.
\newblock Resolution guarantees in electrical impedance tomography.
\newblock {\em IEEE transactions on medical imaging}, 34(7):1513--1521, 2015.

\bibitem{ikehata1998size}
M.~Ikehata.
\newblock Size estimation of inclusion.
\newblock {\em J. Inverse Ill-Posed Probl.}, 6(2):127--140, 1998.

\bibitem{kang1997inverse}
H.~Kang, J.~K. Seo, and D.~Sheen.
\newblock The inverse conductivity problem with one measurement: stability and
  estimation of size.
\newblock {\em SIAM J. Math. Anal.}, 28(6):1389--1405, 1997.

\bibitem{oh2007calibration}
T.~I. Oh, K.~H. Lee, S.~M. Kim, H.~Koo, E.~J. Woo, and D.~Holder.
\newblock Calibration methods for a multi-channel multi-frequency eit system.
\newblock {\em Physiological measurement}, 28(10):1175, 2007.

\bibitem{oh2011fully}
T.~I. Oh, H.~Wi, D.~Y. Kim, P.~J. Yoo, and E.~J. Woo.
\newblock A fully parallel multi-frequency eit system with flexible electrode
  configuration: Khu mark2.
\newblock {\em Physiological measurement}, 32(7):835, 2011.

\bibitem{oh2007multi}
T.~I. Oh, E.~J. Woo, and D.~Holder.
\newblock Multi-frequency eit system with radially symmetric architecture: Khu
  mark1.
\newblock {\em Physiological measurement}, 28(7):S183, 2007.

\bibitem{Oh07a}
T.~I. Oh, E.~J. Woo, and D.~Holder.
\newblock Multi-frequency \uppercase{EIT} system with radially symmetric
  architecture: \uppercase{KHU} \uppercase{M}ark1.
\newblock {\em Physiol. Meas.}, 28:S183--S196, 2007.

\bibitem{tamburrino2006monotonicity}
A.~Tamburrino.
\newblock Monotonicity based imaging methods for elliptic and parabolic inverse
  problems.
\newblock {\em Journal of Inverse and Ill-posed Problems}, 14(6):633--642,
  2006.

\bibitem{tamburrino2002new}
A.~Tamburrino and G.~Rubinacci.
\newblock A new non-iterative inversion method for electrical resistance
  tomography.
\newblock {\em Inverse Problems}, 18(6):1809, 2002.

\bibitem{wi2014multi}
H.~Wi, H.~Sohal, A.~L. McEwan, E.~J. Woo, and T.~I. Oh.
\newblock Multi-frequency electrical impedance tomography system with automatic
  self-calibration for long-term monitoring.
\newblock {\em Biomedical Circuits and Systems, IEEE Transactions on},
  8(1):119--128, 2014.

\bibitem{zhou2015monotonicity}
L.~Zhou, B.~Harrach, and J.~K. Seo.
\newblock Monotonicity-based electrical impedance tomography lung imaging.
\newblock {\em preprint}, 2015.

\end{thebibliography}
\end{document}